\documentclass[article]{siamart1116}

\usepackage{lipsum}
\usepackage{amsfonts}
\usepackage{graphicx}
\usepackage{epstopdf}
\usepackage{algorithmic}
\ifpdf
  \DeclareGraphicsExtensions{.eps,.pdf,.png,.jpg}
\else
  \DeclareGraphicsExtensions{.eps}
\fi

\numberwithin{theorem}{section}

\newcommand{\TheTitle}{Uniform asymptotic and convergence estimates for the Jin-Xin model under the diffusion scaling} 
\newcommand{\TheAuthors}{Roberta Bianchini}

\headers{Asymptotic estimates for the diffusive Jin-Xin model}{\TheAuthors}

\title{{\TheTitle}\thanks{Submitted to the editors DATE.}}

\author{
  Roberta Bianchini \thanks{Universit\`a degli Studi di Roma "Tor Vergata", via della Ricerca Scientifica 1, I-00133 Rome, Italy - Istituto per le Applicazioni del Calcolo "M. Picone", Consiglio Nazionale delle Ricerche, via dei Taurini 19, I-00185 Rome, Italy.}
  }

\usepackage{amsopn}

\ifpdf
\hypersetup{
  pdftitle={\TheTitle},
  pdfauthor={\TheAuthors}
}
\fi

\newtheorem{remark}{Remark}[section]

\begin{document}

\maketitle

\begin{abstract}
We provide sharp decay estimates in time in the context of Sobolev spaces, for smooth solutions to the one dimensional Jin-Xin model under the diffusion scaling, which are uniform with respect to the singular parameter of the scaling. This provides convergence to the limit nonlinear parabolic equation both for large time, and for the vanishing singular parameter. The analysis is performed by means of two main ingredients. First, a crucial change of variables highlights the dissipative property of the Jin-Xin system, and allows to observe a faster decay of the dissipative variable with respect to the conservative one, which is essential in order to close the estimates. Next, the analysis relies on a deep investigation on the Green function of the linearized Jin-Xin model, depending on the singular parameter,  combined with the Duhamel formula in order to handle the nonlinear terms.
\end{abstract}



\section{Introduction}
We consider the following scaled version of the Jin-Xin approximation for systems of conservation laws in \cite{XinJin}:
\begin{equation}
\label{scaled_JinXin}
\begin{cases}
\partial_t u + \partial_x v=0, \\
\varepsilon^2 \partial_t v+ \lambda^2 \partial_x u=f(u)-v,
\end{cases}
\end{equation}
where $\lambda>0$ is a positive constant, $u, v$ depend on $(t, x) \in \mathbb{R}^+\times \mathbb{R}$ and  take values in $\mathbb{R}$, while $f(u): \mathbb{R} \rightarrow \mathbb{R}$ is a Lipschitz function such that $f(0)=0,$ and $f'(0)=a,$ with $a$ a constant value.
The diffusion limit of this system for $\varepsilon \rightarrow 0$  has been studied in \cite{Jin, BGN}, where the convergence to the following equations is proved:
\begin{equation}
\label{limit_JinXin}
\begin{cases}
\partial_t u + \partial_x v=0 \\
v=f(u)-\lambda^2 \partial_x u.
\end{cases}
\end{equation}
From \cite{Natalini, BGN}, it is well-known that system (\ref{scaled_JinXin}) can be written in BGK formulation, \cite{Bouchut}, by means of the linear change of variables:
\begin{equation}
\label{change_BGK}
u=f_1^\varepsilon+f_2^\varepsilon, \qquad v=\dfrac{\lambda}{\varepsilon}(f_1^\varepsilon-f_2^\varepsilon).
\end{equation}
Precisely, the BGK form of (\ref{scaled_JinXin}) reads:
\begin{equation}
\label{BGK_scaled_JinXin}
\begin{cases}
\partial_t f_1^\varepsilon+\dfrac{\lambda}{\varepsilon}\partial_x f_1^\varepsilon=\dfrac{1}{\varepsilon^2}(M_1(u)-f_1^\varepsilon), \\
\partial_t f_2^\varepsilon-\dfrac{\lambda}{\varepsilon}\partial_x f_2^\varepsilon=\dfrac{1}{\varepsilon^2}(M_2(u)-f_2^\varepsilon), \\
\end{cases}
\end{equation}
where the so-called {Maxwellians} are:
\begin{equation}\label{Maxwellians_JinXin}
M_1(u)=\dfrac{u}{2}+\dfrac{\varepsilon f(u)}{2\lambda}, \qquad M_2(u)=\dfrac{u}{2}-\dfrac{\varepsilon f(u)}{2\lambda}.
\end{equation}
According to the theory on diffusive limits of the Boltzmann equation and related BGK models, see \cite{Laure, Laure1}, we take some fluctuations of the Maxwellian functions as initial data for the Cauchy problem associated with system (\ref{scaled_JinXin}). Namely, given a function $\bar{u}_0(x),$ depending on the spatial variable, we assume
\begin{equation}
\label{approx_initial_data_JinXin}
(u(0, x), v(0, x))=(u_0, v_0)= (\bar{u}_0, \; f(\bar{u}_0)-\lambda^2 \partial_x \bar{u}_0),
\end{equation}
indeed perturbations of the Maxwellians, as it is clear by expressing the initial data (\ref{approx_initial_data_JinXin}) through the change of variables (\ref{change_BGK}), i.e.
\begin{equation}
\label{approx_initial_data_JinXin_f}
(f_1^\varepsilon(0, x), f_2^\varepsilon(0, x))=\Bigg(M_1(\bar{u}_0)-\frac{\varepsilon \lambda}{2}\partial_x \bar{u}_0, \; M_2(\bar{u}_0)+\frac{\varepsilon \lambda}{2}\partial_x \bar{u}_0 \Bigg),
\end{equation}
where the fluctuations are given by $\pm \dfrac{\varepsilon \lambda}{2} \partial_x \bar{u}_0.$ 

\indent System (\ref{scaled_JinXin}) is the parabolic scaled version of the hyperbolic relaxation approximation for systems of conservation laws, the Jin-Xin system, introduced in \cite{XinJin} in 1995. This model has been studied in \cite{NataliniCPAM, Chern, XinJin}, and the hyperbolic relaxation limit has been investigated. A complete review on hyperbolic conservation laws with relaxation, and a focus on the Jin-Xin system is presented in \cite{Mascia}. By means of the Chapman-Enskog expansion, local attractivity of diffusion waves for the Jin-Xin model was established in \cite{Chern}. In \cite{Natalini3}, the authors showed that, under some assumptions on the initial data and the function $f(u),$ the first component of system (\ref{scaled_JinXin}) with $\varepsilon=1$ decays asymptotically towards the fundamental solution to the Burgers equation, for the case of $f(u)=\alpha u^2/2.$ Besides, \cite{Zuazua} is a complete study of the long time behavior of this model for a more general class of functions $f(u)=|u|^{q-1}u,$ with $q\ge 2.$ The method developed in \cite{Zuazua} can be also extended to the multidimensional case in space, and provides sharp decay rates. Here we study the parabolic scaled version of the system studied in \cite{Zuazua}, i.e. (\ref{scaled_JinXin}), and we consider a more general function $f(u)=a u + h(u),$ where $a$ is a constant, and $h(u)$ is a quadratic function. We point out that only the case $a=0$ has been handled in \cite{Zuazua}, and in many previous works as well. In accordance with the theory presented in \cite{Bianchini} on partially dissipative hyperbolic systems, we are able to cover also the case $a\neq 0$. Furthermore, besides the aymptotic behavior of the solutions, here we are interested in studying the diffusion limit, for vanishing $\varepsilon,$ of the Jin-Xin system, which is the main improvement of the present paper with respect to the results achieved in \cite{Bianchini}. Indeed, because of the presence of the singular parameter, we cannot approximate the analysis of the Green function of the linearized problems, as the authors did in \cite{Bianchini}, and explicit calculations in that context are needed.\\ The diffusive Jin-Xin system has been already investigated in the following works below.
In \cite{Jin}, initial data around a traveling wave were considered, while in \cite{BGN} the authors write system (\ref{scaled_JinXin}) in terms of a BGK model, and the diffusion limit is studied by using monotonicity properties of the solution. In all these cases, $u, v$ are scalar functions. For simplicity, here we also take scalar unknowns $u, v.$ However, our approach, which takes its roots in \cite{Bianchini}, can be generalized to the case of vectorial functions $u, v \in \mathbb{R}^N$. As mentioned before, the novelty of the present paper consists in dealing with the singular approximation and, in the meanwhile, with the the large time asymptotic of system (\ref{scaled_JinXin}), which behaves like the limit parabolic equation (\ref{limit_JinXin}), without using monotonicity arguments. We obtain, indeed, sharp decay estimates in time to the solution to system (\ref{scaled_JinXin}) in the Sobolev spaces, which are uniform with respect to the singular parameter. This provides the convergence to the limit nonlinear parabolic equation (\ref{limit_JinXin}) both asymptotically in time, and in the vanishing $\varepsilon$-limit. To this end, we perform a crucial change of variables that highlights the dissipative property of the Jin-Xin system, and provides a faster decay of the dissipative variable with respect to the conservative one, which allows to close the estimates. Next, a deep investigation on the Green function of the linearized system (\ref{scaled_JinXin}) and the related spectral analysis is provided, since explicit expressions are needed in order to deal with the singular parameter $\varepsilon$. The dissipative property of the diffusive Jin-Xin system, together with the uniform decay estimates discussed above, and the Green function analysis combined with the Duhamel formula provide our main results, stated in the following. Define
$$E_m=\max\{\|u_0\|_{L^1}+\varepsilon \|v_0-au_0\|_{L^1}, \|u_0\|_m+\varepsilon \|v_0-au_0\|_m\},$$
where $\|\cdot\|_m$ stands for the $H^m(\mathbb{R})$ Sobolev norm and $H^0(\mathbb{R})=L^2(\mathbb{R}).$
\begin{theorem}
\label{uniform_global_existence}
Consider the Cauchy problem associated with system (\ref{scaled_JinXin}) and initial data (\ref{approx_initial_data_JinXin}). If $E_2$ is sufficiently small, 
then the unique solution $$(u, v) \in C([0, \infty), H^2(\mathbb{R})) \cap C^1([0, \infty), H^1(\mathbb{R})).$$ Moreover, the following decay estimate holds:
$$\|u(t)\|_2+\varepsilon \|v(t)-au(t)\|_2 \le C \min\{1, t^{-1/4}\}E_2.$$
\end{theorem}

Now consider the following equation
\begin{equation}
\label{CE_no_epsilon_JinXin}
\partial_t w_p + a \partial_x w_p + \partial_x h(w_p) -\lambda^2\partial_{xx} w_p=0,
\end{equation}

\begin{theorem}
\label{asymptotic_behavior_JinXin}
Let $w_p$ be the solution to the nonlinear equation (\ref{CE_no_epsilon_JinXin}) with sufficiently smooth  initial data 
$$w_p(0)=u(0)={u}_0,$$
where ${u}_0$ in (\ref{approx_initial_data_JinXin}) is the initial datum for the Jin-Xin system (\ref{scaled_JinXin}). For any $\mu \in [0, 1/2),$ if $E_1$ is sufficiently small with respect to $(1/2-\mu),$ then we have the following decay estimate:
\begin{equation}
\|D^\beta (u(t)-w_p(t))\|_0 \le C \varepsilon \min \{1, t^{-1/4-\mu -\beta/2}\} E_{|\beta|+4},
\end{equation}
with $C=C(E_{|\beta|+\sigma})$ for $\sigma $ large enough.
\end{theorem}

Once we identified the right scaled variable to study system (\ref{scaled_JinXin}), ($u, \varepsilon^2 v$), which are expressed at the beginning of Section \ref{Section1}, and we found the strategy, disscussed in Section \ref{Section1} and \ref{Section2}, to achieve the so-called \emph{conservative-dissipative} (C-D) form in \cite{Bianchini} for our model, our approach essentially relies on the method developed in \cite{Bianchini}, with substantial differences listed here.
\begin{itemize}
\item We need an explicit Green function analysis of the linearized system rather than expansions and approximations, in order to deal with the singular parameter $\varepsilon$. The analysis performed in first part of Section \ref{Section3} is as precise as it is possible.
\item Some estimates in \cite{Bianchini} rely on the use of the Shizuta-Kawashima (SK) condition, explained in the following. Consider a linear first order system in compact form: $\partial_t \textbf{u}+A\partial_x \textbf{u}=G\textbf{u}.$ Passing to the Fourier transform, define $E(i \xi)=G-iA\xi.$ The (SK) condition states that, if $\lambda(z)$ is an eigenvalue of $E(z),$ then ${Re}(\lambda(i\xi)) \le -c \frac{|\xi|^2}{1+|\xi|^2},$ for some constant $c>0$ and for every $\xi \in \mathbb{R} - \{0\}$. As it can be seen in (\ref{Eig_SK}), these eigenvalues for the
compact linearized system in (C-D) form (\ref{CD_real_JinXin_no_tilde}) of system (\ref{scaled_JinXin}) have different weights in $\varepsilon$. Thus, we cannot simply apply the (SK) condition to estimate the remainders in paragraph \emph {Remainders in between} as the authors did in \cite{Bianchini}, since the weights in $\varepsilon$ are essential to deal with the singular nonlinear term in the Duhamel formula (\ref{Duhamel}). Again, a further analysis is needed.
\item Differently from \cite{Bianchini}, we are not assuming to have a global in time solution, uniformly bounded in $\varepsilon$, for our singular system. The uniform global existence is proved in Theorem \ref{uniform_global_existence}.
\item The coupling between the convergence to the limit equation (\ref{limit_JinXin}) for vanishing $\varepsilon$ and for large time in the last section is the main novelty of the present paper, and new ideas are needed to get this result.
\end{itemize}


\section{General setting}\label{Section1}
First of all, we write system (\ref{scaled_JinXin}) in the following form:
\begin{equation}
\label{scaled_JinXin_new_variables}
\begin{cases}
\partial_t u + \dfrac{\partial_x(\varepsilon^2 v)}{\varepsilon^2}=0, \\
\partial_t(\varepsilon^2 v)+ \lambda^2 \partial_x u=f(u)-\dfrac{\varepsilon^2 v}{\varepsilon^2}.
\end{cases}
\end{equation}
The unknown variable is $\textbf{u}=(u, \varepsilon^2 v),$ 
in the spirit of the scaled variables introduced in \cite{Bianchini4}, which are the right scaling to get the conservative-dissipative form discussed below.
Here we write $f(u)=au+h(u),$ where $a=f'(0),$ and system (\ref{scaled_JinXin_new_variables}) reads
\begin{equation}
\label{scaled_JinXin_new_variables_new}
\begin{cases}
\partial_t u + \dfrac{\partial_x(\varepsilon^2 v)}{\varepsilon^2}=0, \\
\partial_t(\varepsilon^2 v)+ \lambda^2 \partial_x u=au + h(u)-\dfrac{\varepsilon^2 v}{\varepsilon^2}.
\end{cases}
\end{equation}

Equations (\ref{scaled_JinXin_new_variables_new}) can be written in compact form:
\begin{equation}
\label{JinXin_compact}
\partial_t \textbf{u} + A \partial_x \textbf{u} = -B \textbf{u}+N(u),
\end{equation}
where 
\begin{equation}
\label{matrices_Jin_Xin_compact_form}
A=\left(\begin{array}{cc}
0 & \dfrac{1}{\varepsilon^2} \\
\lambda^2 & 0 \\
\end{array}\right), \qquad 
-B=\left(\begin{array}{cc}
0 & 0 \\
a & -\dfrac{1}{\varepsilon^2}
\end{array}\right),  \qquad
N(u)=\left(\begin{array}{c}
0 \\
h(u)
\end{array}\right).
\end{equation}
In particular, $-B\textbf{u}$ is the linear part of the source term, while $N(u)$ is the remaining nonlinear one, which only depends on the first component of $\textbf{u}=(u, \varepsilon^2 v)$. 
Now, we look for a right constant symmetrizer $\Sigma $ for system (\ref{JinXin_compact}), which also highlights the dissipative properties of the linear source term. Thus, we find
\begin{equation}
\label{symmetrizer_JinXin}
\Sigma=\left(\begin{array}{cc}
1 & a \varepsilon^2 \\
a \varepsilon^2 & \lambda^2 \varepsilon^2 \\
\end{array}\right).
\end{equation}
Taking $\textbf{w}$ such that
\begin{equation}
\label{change_JinXin}
\textbf{u}=\left(\begin{array}{c}
u \\
\varepsilon^2 v
\end{array}\right)=\Sigma \textbf{w}=\left(\begin{array}{c}
(\Sigma \textbf{w})_1 \\
(\Sigma \textbf{w})_2
\end{array}\right), \;\; \text{where} \;\; \textbf{w}=\left(\begin{array}{c}
w_1 \\
w_2
\end{array}\right)
=\left(\begin{array}{c}
\dfrac{u \lambda^2-a \varepsilon^2 v}{\lambda^2 - a^2 \varepsilon^2} \\\\
\dfrac{v-au}{\lambda^2-a^2 \varepsilon^2}
\end{array}\right),
\end{equation}
system (\ref{JinXin_compact}) reads
\begin{equation}
\label{JinXin_compact_new_variables}
\Sigma \partial_t \textbf{w}+A_1\partial_x \textbf{w}=-B_1\textbf{w}+N((\Sigma \textbf{w})_1),
\end{equation}
where
\begin{equation*}
A_1=A_1^T=A\Sigma=\left(\begin{array}{cc}
a & \lambda^2 \\
\lambda^2 & a \lambda^2 \varepsilon^2
\end{array}\right),\qquad 
-B_1=-B\Sigma=\left(\begin{array}{cc}
0 & 0 \\
0 & a^2 \varepsilon^2 - \lambda^2
\end{array}\right), 
\end{equation*}
\begin{equation}
\label{matrices_JinXin_new_variables}
N((\Sigma \textbf{w})_1)=\left(\begin{array}{c}
0 \\
h(w_1+a\varepsilon^2 w_2)
\end{array}\right).
\end{equation}

By using the Cauchy inequality we get the following lemma.
\begin{lemma}
\label{Equivalence_scalar_product}
The symmetrizer $\Sigma$ is definite positive. Precisely
\begin{equation}
\label{Equivalence_scalar_product_equation}
\dfrac{1}{2}\|w_1\|_0^2  + \varepsilon^2\|w_2\|_0^2 (\lambda^2-2a^2 \varepsilon^2) \le (\Sigma \textbf{w}, \textbf{w})_0 \le \|w_1\|_0^2 (1+a\varepsilon^2) + \|w_2\|_0^2 (a+\lambda^2)\varepsilon^2.
\end{equation}
\end{lemma}

Notice that from the theory on hyperbolic systems, \cite{Majda}, the Cauchy problem for (\ref{JinXin_compact_new_variables}) with initial data $\textbf{w}_0$ in $H^m(\mathbb{R}), \, m \ge 2,$ has a unique local smooth solution $\textbf{w}^\varepsilon$ for each fixed $\varepsilon > 0$. We denote by $T^\varepsilon$ the maximum time of existence of this local solution and, hereafter, we consider the time interval $[0, T^*],$ with $T^* \in [0, T^\varepsilon)$ for every $\varepsilon$.
In the following, we study the Green function of system (\ref{JinXin_compact_new_variables}), and we establish some uniform energy estimates and decay rates of the smooth solution to system (\ref{JinXin_compact_new_variables}).

\subsection{The conservative-dissipative form}
In this section, we introduce a linear change of variable, so providing a particular structure for our system, the so-called \emph{conservative-dissipative form} (C-D) defined in \cite{Bianchini}. The (C-D) form allows to identify a conservative variable and a dissipative one for system (\ref{scaled_JinXin}), such that in the following a crucial faster decay of the dissipative variable is observed. Thanks to this change of variables, we are able indeed to handle the case $a\neq 0$ in (\ref{scaled_JinXin_new_variables_new}).
Hereafter, $(\cdot, \cdot)$ denotes the standard scalar product in $L^2(\mathbb{R}),$ and $\| \cdot \|_m$ is the $H^m(\mathbb{R})$-norm, for $m \in \mathbb{N},$ where $H^0(\mathbb{R})=L^2(\mathbb{R}).$
\begin{proposition} 
Given the right symmetrizer $\Sigma$ in (\ref{symmetrizer_JinXin}) for system (\ref{JinXin_compact}), denoting by
\begin{equation}
\label{change_CD_JinXin}
\tilde{\textbf{w}}=M\textbf{u}=\left(\begin{array}{cc}
1 & 0 \\\\
\dfrac{-a \varepsilon}{\sqrt{\lambda^2 - a^2 \varepsilon^2}} & \dfrac{1}{\varepsilon \sqrt{\lambda^2-a^2 \varepsilon^2}}
\end{array}\right) \textbf{u}=\left(\begin{array}{c}
u \\\\
\dfrac{\varepsilon(v-au)}{\sqrt{\lambda^2-a^2\varepsilon^2}}
\end{array}\right),
\end{equation}
system (\ref{JinXin_compact}) can be written in (C-D) form defined in \cite{Bianchini}, i.e.
\begin{equation}
\label{CD_real_JinXin}
\partial_t \tilde{\textbf{w}} + \tilde{A}\partial_x \tilde{\textbf{w}} = -\tilde{B} \tilde{\textbf{w}} + \tilde{N}(\tilde{w}_1), 
\end{equation}
where 
\begin{equation}
\label{matrices_CD_real_JinXin}
\small{
\begin{aligned}
\tilde{A}=\left(\begin{array}{cc}
a & \dfrac{\sqrt{\lambda^2-a^2\varepsilon^2}}{\varepsilon} \\\\
\dfrac{\sqrt{\lambda^2-a^2\varepsilon^2}}{\varepsilon} & {-a}
\end{array}\right), \; 
\tilde{B}=\left(\begin{array}{cc}
0 & 0 \\\\
0 & \dfrac{1}{\varepsilon^2}
\end{array}\right), \; \tilde{N}(\tilde{w}_1)=\left(\begin{array}{c}
0 \\\\
\dfrac{h(\tilde{w}_1)}{\varepsilon\sqrt{\lambda^2-a^2\varepsilon^2}}
\end{array}\right).
\end{aligned}}
\end{equation}
\end{proposition}

\section{The Green function of the linear partially dissipative system} \label{Section2}
We consider the linear part of the (C-D) system (\ref{CD_real_JinXin})-(\ref{matrices_CD_real_JinXin}) without the \emph{tilde} for simplicity,
\begin{equation}
\label{CD_real_JinXin_no_tilde}
\partial_t \textbf{w}+A\partial_x \textbf{w}=-B\textbf{w}.
\end{equation}
We want to apply the approach developed in \cite{Bianchini}, to study the singular approximation system above. The main difficulty here is to deal with the singular perturbation parameter $\varepsilon$.
We consider the Green kernel $\Gamma(t,x)$ of (\ref{CD_real_JinXin_no_tilde}), which satisfies
\begin{equation}
\label{Green_kernel_JinXin}
\begin{cases}
\partial_t \Gamma + A \partial_x \Gamma = -B\Gamma, \\
\Gamma (0, x)= \delta(x) I.
\end{cases}
\end{equation}
Taking the Fourier transform $\hat{\Gamma},$ we get
\begin{equation}
\label{Fourier_JinXin}
\begin{cases}
\frac{d}{dt}\hat{\Gamma}=(-B-i\xi A) \hat{\Gamma}, \\
\hat{\Gamma}(0, \xi)=I.
\end{cases}
\end{equation}
Consider the entire function
\begin{equation}
\label{E_JinXin}
E(z)=-B-zA=\left(\begin{array}{cc}
-az & -\dfrac{z\sqrt{\lambda^2-a^2\varepsilon^2}}{\varepsilon} \\\\
-\dfrac{z\sqrt{\lambda^2-a^2\varepsilon^2}}{\varepsilon} & {az}-\dfrac{1}{\varepsilon^2}
\end{array}\right).
\end{equation}
Formally, the solution to (\ref{Fourier_JinXin}) is given by
\begin{equation}
\label{Formal_solution_JinXin}
\hat{\Gamma}(t, \xi)=e^{E(i\xi)t}=\sum_{n=0}^\infty (-B -i\xi A)^n.
\end{equation}
Since $E(z)$ in (\ref{E_JinXin}) is symmetric, if $z$ is not exceptional we can write
$$E(z)=\lambda_1(z)P_1(z)+\lambda_2(z)P_2(z),$$
where $\lambda_1(z), \lambda_2(z)$ are the eigenvalues of $E(z),$ and $P_1(z), P_2(z)$ the related eigenprojections, given by
$$P_j(z)=-\frac{1}{2\pi i} \oint_{|\xi-\lambda_j(z)|<<1} (E(z)-\xi I)^{-1} \, d\xi, \qquad j=1, 2.$$
Following \cite{Bianchini}, we study the low frequencies (case $z=0$) and the high frequencies (case $z=\infty$) separately.
\paragraph{\textbf{Case $z=0$}}
The total projector for the eigenvalues near to $0$ is
\begin{equation}
\label{total_projector_near_zero}
P(z)=-\dfrac{1}{2 \pi i} \oint_{|\xi|<<1} (E(z)-\xi I)^{-1} \, d\xi.
\end{equation}
Besides, it has the following expansion, see \cite{Kato},
\begin{equation}
\label{total_projector_expansion}
P(z)=P_0+\sum_{n \ge 1} z^{n}P^{n}(z),
\end{equation}
where $P_0$ is the eigenprojection for $E(0)-\xi I=-B-\xi I,$ i.e.
\begin{equation}
\label{Pn_JinXin}
P_0=-\dfrac{1}{2 \pi i } \oint_{|\xi| <<1} (-B-\xi I)^{-1} \, d\xi=:Q_0=\left(\begin{array}{cc}
1 & 0 \\
0 & 0
\end{array}\right), \; P^{n}(z)=-\dfrac{1}{2\pi i} \oint R^{(n)} (\xi) \, d\xi,
\end{equation}
with $R^{(n)}$ the $n$-th term in the expansion of the resolvent (\ref{Rn_XinJin}). Here $Q_0$ is the projection onto the null space of the source term, while we denote by $Q_{-}=I-Q_0$ the complementing projection, and by $L_{-}, L_0$ and $R_{-}, R_0$  the related left and right eigenprojectors, see \cite{Kato, Bianchini}, i.e.
$$L_{-}=R_{-}^T=\left(\begin{array}{cc}
0 & 1
\end{array}\right), \qquad L_0=R_0^T=\left(\begin{array}{cc}
1 & 0
\end{array}\right),$$
$$Q_{-}=R_{-}L_{-}, \qquad Q_0=R_0L_0.$$ 
On the other hand, from \cite{Kato},
\begin{align*}
R(\xi, z)=(E(z)-\xi I)^{-1} & =(-B-zA-\xi I)^{-1}=(-B-\xi I)^{-1} \sum_{n=0}^\infty (Az (-B-\xi I)^{-1})^n\\
&=(-B-\xi I)^{-1}+\sum_{n \ge 1} (-B-\xi I)^{-1} z^n (A (-B-\xi I)^{-1})^n\\
&=R_0 (\xi)+\sum_{n \ge 1} R^{(n)} (\xi),
\end{align*}
i.e.
\begin{equation}
\label{Rn_XinJin}
R^{(n)}=z^n (-B-\xi I)^{-1} (A (-B-\xi I)^{-1})^n.
\end{equation}
Since a neighborhood of $z=0$ is considered, at this point the authors in \cite{Bianchini} take the first two terms of the asymptotic expansion of the total projector (\ref{total_projector_expansion}), so obtaining an expression with a remainder $O(z^2)$. We cannot approximate the projector in the same way, since we need to check the singular terms in $\varepsilon$. Thus, we perform an explicit spectral analysis for the Green function of our problem. First of all,
\begin{equation}
A(-B-\xi I)^{-1}=\left(\begin{array}{cc}
-\dfrac{a}{\xi} & -\dfrac{\varepsilon \sqrt{\lambda^2-a^2\varepsilon^2}}{1+\varepsilon^2 \xi} \\\\
-\dfrac{\sqrt{\lambda^2-a^2\varepsilon^2}}{\varepsilon \xi} & \dfrac{a\varepsilon^2}{1+\varepsilon^2 \xi} 
\end{array}\right),
\end{equation}
which is diagonalizable, i.e.
$$A(-B-\xi I)^{-1}=V D V^{-1},$$
where $D$ is the diagonal matrix with entries given by the eigenvalues, and $V$ is the matrix with the eigenvectors on the columns. Explicitly, setting
\begin{equation}
\label{Box_JinXin}
\Box:=a^2+4\varepsilon^2 \lambda^2 \xi^2 + 4 \lambda^2 \xi,
\end{equation}
we have
\begin{equation}
\label{eigenvalues_vectors_XinJin}
D=diag \Bigg\{ \dfrac{-a \pm \sqrt{\Box}}{2 \xi (1+\varepsilon^2 \xi)}\Bigg\}, \, V=\left(\begin{array}{cc}
\displaystyle \frac{\varepsilon (a+\sqrt{\Box}+2a \varepsilon^2 \xi)}{2(1+\varepsilon^2 \xi) \sqrt{\lambda^2-a^2 \varepsilon^2}} & \displaystyle \frac{\varepsilon (a-\sqrt{\Box}+2a \varepsilon^2 \xi)}{2(1+\varepsilon^2 \xi) \sqrt{\lambda^2-a^2 \varepsilon^2}} \\\\
1 & 1
\end{array}\right),
\end{equation}
\begin{equation}
\label{inv_eigenvectors_JinXin}
V^{-1}=\left(\begin{array}{cc}
\dfrac{(1+\varepsilon^2 \xi) \sqrt{\lambda^2-a^2\varepsilon^2}}{\varepsilon \sqrt{\Box}} & \dfrac{-a+\sqrt{\Box}-2a \varepsilon^2 \xi}{2 \sqrt{\Box}} \\\\
-\dfrac{(1+\varepsilon^2 \xi) \sqrt{\lambda^2-a^2\varepsilon^2}}{\varepsilon \sqrt{\Box}} & \dfrac{a+\sqrt{\Box}+2a \varepsilon^2 \xi}{2 \sqrt{\Box}}
\end{array}\right).
\end{equation}
This way, denoting by
$$\Diamond_1=a-\sqrt{\Box}+2a\varepsilon^2 \xi, \; \Diamond_2=a+\sqrt{\Box}+2a\varepsilon^2 \xi,  \; \triangle_1=-\dfrac{a+\sqrt{\Box}}{2\xi (1+\varepsilon^2 \xi)}, \; \triangle_2=-\dfrac{a-\sqrt{\Box}}{2 \xi (1+\varepsilon^2 \xi)},$$
with $\Box$ in (\ref{Box_JinXin}), from (\ref{Rn_XinJin}) we have
$$R^{(n)}=z^n (-B-\xi I)^{-1}(A(-B-\xi I)^{-1})^n=z^n (-B-\xi I)^{-1}(V D^n V^{-1})$$
$$=z^n\left(\begin{array}{cc}
\dfrac{\Diamond_1 \triangle_2^n-\Diamond_2 \triangle_1^n}{2\sqrt{\Box} \xi} & -\varepsilon \sqrt{\dfrac{\lambda^2-a^2\varepsilon^2}{\Box}} (\triangle_1^n-\triangle_2^n) \\\\
-\varepsilon \sqrt{\dfrac{\lambda^2-a^2\varepsilon^2}{\Box}} (\triangle_1^n-\triangle_2^n) &
-\dfrac{\varepsilon^2}{2(1+\varepsilon^2 \xi)\sqrt{\Box}}(\Diamond_2 \triangle_2^n-\Diamond_1 \triangle_1^n)
\end{array}\right).$$
The matrix above is completely bounded in $\varepsilon,$ and so we can approximate the expression of the total projector (\ref{total_projector_expansion}) up to the second order. To this end, we consider the previous expression of $R^{(n)}$ for $n=0, 1, 2,$ we apply the integral formula (\ref{Pn_JinXin}) and we obtain
\begin{equation}
\label{total_projector_approximation}
P(z)=\left(\begin{array}{cc}
1+O(z^2) & -\varepsilon z \sqrt{\lambda^2-a^2\varepsilon^2} + \varepsilon O(z^2) \\\\
-\varepsilon z \sqrt{\lambda^2-a^2\varepsilon^2}  + \varepsilon O(z^2) \quad & \varepsilon^2 z^2(\lambda^2-a^2\varepsilon^2) + \varepsilon^2 O(z^3)
\end{array}\right).
\end{equation}
Now, we consider the left $L(z)$ and the right $R(z)$ eigenprojectors of $P(z),$ i.e. 
\begin{align*}
& P(z)=R(z)L(z), \quad L(z)R(z)=I \\
& L(z)P(z)=L(z), \quad P(z)R(z)=R(z).
\end{align*}
We can limit ourselves to the second order approximation, according to (\ref{total_projector_approximation}). Then, we consider 
\begin{equation}
\label{tilde_Projector}
\tilde{P}(z)=\left(\begin{array}{cc}
1& -\varepsilon z \sqrt{\lambda^2-a^2\varepsilon^2} \\\\
-\varepsilon z \sqrt{\lambda^2-a^2\varepsilon^2}  \quad & \varepsilon^2 z^2(\lambda^2-a^2\varepsilon^2) 
\end{array}\right),
\end{equation}
and, by applying the conditions above, we obtain
\begin{align*}
\tilde{L}(z)=\left(\begin{array}{cc} 
1 & -\varepsilon z \sqrt{\lambda^2-a^2 \varepsilon^2}
\end{array}\right), \qquad \tilde{R}(z)=\left(\begin{array}{c}
1 \\
-\varepsilon z \sqrt{\lambda^2-a^2\varepsilon^2}
\end{array}\right),
\end{align*}
where 
\begin{equation}
\begin{array}{cc}
\tilde{P}(z)=\tilde{R}(z)\tilde{L}(z), \quad & \tilde{L}(z)\tilde{R}(z)=1+\varepsilon^2 O(z^2), \\
\tilde{P}(z)\tilde{R}(z)=\tilde{R}(z)+\varepsilon^2 O(z^2) , \quad &
\tilde{L}(z)\tilde{P}(z)=\tilde{L}(z)+\varepsilon^2 O(z^2) ,
\end{array}
\end{equation}
and so $$P(z)=\tilde{P}(z)+O(z^2), \quad R(z)=\tilde{R}(z)+O(z^2), \quad \quad L(z)=\tilde{L}(z)+O(z^2).$$
Let us point out that further expansions of $L(z), R(z)$ are not singular in $\varepsilon$ too, since the weights in $\varepsilon$ of these vectors come from (\ref{total_projector_approximation}). Precisely, one can see that $L^\varepsilon(z)$ depends on $\varepsilon$ as follows:
$$L^\varepsilon(\cdot)=\left(\begin{array}{cc}
1 \quad & O(\varepsilon)
\end{array}\right)=[{R(\cdot)^\varepsilon}]^T.$$
Now, by using the left and the right operators, we decompose $E(z)$ in the following way, see \cite{Bianchini},
\begin{equation}
\label{E_decompostition_JinXin}
E(z)=R(z)F(z)L(z)+R_{-}(z)F_{-}(z)L_{-}(z),
\end{equation}
where $L_{-}(z), R_{-}(z)$ are left and right eigenprojectors of 
$P_{-}(z)=I-P(z),$ while
$$F(z)=L(z)E(z)R(z), \quad F_{-}(z)=L_{-}(z)E(z)R_{-}(z).$$
We use the approximations of $L(z), R(z)$ above, and so
\begin{equation}
\label{F_approximation_JinXin}
F(z)=(\tilde{L}(z)+O(z^2))(-B-Az)(\tilde{R}(z)+O(z^2))=-az+(\lambda^2-a^2\varepsilon^2)z^2+O(z^3).
\end{equation}
We study $F_{-}(z)$. Matrix (\ref{total_projector_approximation}) and the definition above imply that
\begin{equation}
\label{Compl_total_projector_JinXin}
P_{-}(z)=\left(\begin{array}{cc}
O(z^2) \quad & z \varepsilon \sqrt{\lambda^2-a^2 \varepsilon^2}+\varepsilon O(z^2) \\\\
z \varepsilon \sqrt{\lambda^2-a^2 \varepsilon^2}+\varepsilon O(z^2) \quad & 1+\varepsilon^2O(z^2)
\end{array}\right),
\end{equation}
and, approximating again,
$$L_{-}(z)=\tilde{L}_{-}(z)+O(z^2)=\left(\begin{array}{cc}
z \varepsilon \sqrt{\lambda^2-a^2\varepsilon^2} \quad & 1
\end{array}\right)+O(z^2),$$ 
$$R_{-}(z)=\tilde{R}_{-}(z)+O(z^2)=\left(\begin{array}{c}
z \varepsilon \sqrt{\lambda^2-a^2\varepsilon^2} \\
1
\end{array}\right)+O(z^2).$$
Thus,
\begin{equation}
\label{F_meno_approximation_JinXin}
F_{-}(z)=\tilde{L}_{-}(z)(-B-Az)\tilde{R}_{-}(z)+O(z^2)=-\frac{1}{\varepsilon^2}+az+O(z^2).
\end{equation}

This yields the proposition below.
\begin{proposition}
\label{Proposition_z_small}
We have the following decomposition near $z=0$:
\begin{equation}
\label{Decomposition_z_small}
E(z)=F(z)P(z)+E_{-}(z),
\end{equation}
with $F(z)$ in (\ref{F_approximation_JinXin}), $P(z)$ in (\ref{total_projector_approximation}), $E_{-}(z)=R_{-}(z)F_{-}(z)L_{-}(z),$ and $F_{-}(z)$ in (\ref{F_meno_approximation_JinXin}).
\end{proposition}

\paragraph{\textbf{Case $z=\infty$}}
We consider
$E(z)=-B-Az=z(-B/z-A)=z E_1(1/z)$ and, setting $z=i\xi$ and  $\zeta=1/z=-i\eta,$ with $\xi, \eta \in \mathbb{R},$ we have $E_1(\zeta)=-A- \zeta B$
$$=\left(\begin{array}{cc}
-a \quad & -\dfrac{\sqrt{\lambda^2-a^2 \varepsilon^2}}{\varepsilon} \\\\
 -\dfrac{\sqrt{\lambda^2-a^2 \varepsilon^2}}{\varepsilon} \quad & a- \dfrac{\zeta}{\varepsilon^2}
\end{array}\right)
=\left(\begin{array}{cc}
-a \quad & -\dfrac{\sqrt{\lambda^2-a^2 \varepsilon^2}}{\varepsilon} \\\\
 -\dfrac{\sqrt{\lambda^2-a^2 \varepsilon^2}}{\varepsilon} \quad & a+ \dfrac{i \eta}{\varepsilon^2}
\end{array}\right).$$
Since $E_1(\zeta)$ is symmetric, we determine the eigenvalues and the right eigenprojectors, 
$$-A-\zeta B=\lambda_1^{E_1}(\zeta) R_1(\zeta) R_1^T(\zeta) + \lambda_2^{E_1}(\zeta) R_2(\zeta) R_2^T(\zeta),$$
such that, for $j=1, 2,$ $R_j^T(\zeta) R_j(\zeta)=I$.
The following expression for the eigenvalues of $E_1(i \eta)$ is provided
$$\lambda_{1, 2}^{E_1}(z)= \dfrac{i \eta}{2 \varepsilon^2} \pm \dfrac{\sqrt{4 \varepsilon^2 \lambda^2 + 4 a \eta \varepsilon^2 i - \eta^2}}{2 \varepsilon^2},$$ and it is simple to prove that both the corresponding eigenvalues of $E(z),$ which can be obtained multiplying $\lambda_1^{E_1}(z)$ and $\lambda_2^{E_1}(z)$ above by $z=i\xi=i/\eta,$
have a strictly negative real part in the high frequencies regime ($|\zeta|=|\eta| << 1$) and in the vanishing $\varepsilon$ limit. Moreover, setting 
$\delta_{1, 2}=\sqrt{8 \varepsilon^2 \lambda^2 + 2 \zeta^2 - 8a\varepsilon^2 \zeta \pm (- 2 \zeta \sqrt{\mu}+4a\varepsilon^2 \sqrt{\mu})},$ where $\mu=4\varepsilon^2 \lambda^2+\zeta^2-4a\varepsilon^2\zeta,$ the normalized right eigenprojectors are given by:
$$R_1(\zeta)=\dfrac{1}{\delta_1}\left(\begin{array}{c}
{(2a\varepsilon^2-\zeta)+\sqrt{\mu}} \\\\
{2\varepsilon \sqrt{\lambda^2-a^2\varepsilon^2}}
\end{array}\right), \quad R_2(\zeta)=\dfrac{1}{\delta_2}\left(\begin{array}{c}
{(2a\varepsilon^2-\zeta)-\sqrt{\mu}} \\\\
{2\varepsilon \sqrt{\lambda^2-a^2\varepsilon^2}}
\end{array}\right).$$

The eigenprojectors are bounded in $\varepsilon,$ even for $\zeta$ near to zero. Thus, we can approximate the total projector of
$E_1(\zeta)=-A-\zeta B$ in a more convenient way, i.e. we decompose
$$A=\lambda_1 R_1 R_1^T+\lambda_2 R_2 R_2^T,$$ where
$\lambda_1=\lambda/\varepsilon, \, \lambda_2=-\lambda/\varepsilon,$ and the corresponding eigenprojectors $$ R_1=\dfrac{1}{\sqrt{2\lambda}}\left(\begin{array}{c}
\sqrt{\dfrac{\lambda^2-a^2 \varepsilon^2}{(\lambda - a \varepsilon)}} \\\\
\sqrt{{\lambda-a \varepsilon}}
\end{array}\right),
\; R_2=\dfrac{1}{\sqrt{2\lambda}}\left(\begin{array}{c}
-\sqrt{\dfrac{\lambda^2-a^2 \varepsilon^2}{(\lambda+a \varepsilon)}} \\\\
\sqrt{{\lambda+a \varepsilon}}
\end{array}\right).$$
Now, by considering the total projector for the family of eigenvalues going to $\lambda_j=\pm \lambda/\varepsilon$ as $\zeta \approx 0,$ we obtain the following approximations:
\begin{equation}
\label{F1_JinXin}
{F_1}_j(\zeta)=-\lambda_j I + \zeta R_j^T (-B) R_j + O(\zeta^2).
\end{equation}
Explicitly,
\begin{equation}
\label{F12_JinXin_explicit}
{F_1}_1(\zeta)=-\dfrac{\lambda}{\varepsilon} - \dfrac{(\lambda-a \varepsilon)\zeta}{2\lambda \varepsilon^2} + O(\zeta^2), \;\;\; {F_1}_2(\zeta)=\dfrac{\lambda}{\varepsilon} - \dfrac{(\lambda+a \varepsilon)\zeta}{2\lambda \varepsilon^2} + O(\zeta^2).
\end{equation}

Since $E(z)=z E_1(1/z),$ we multiply $F_1(\zeta)=F_1(1/z)$ by $z$ and, for $|z|\rightarrow + \infty,$
\begin{equation}
\label{eig_zeta_infty_JinXin}
\lambda_1(z)=-\dfrac{\lambda}{\varepsilon}z-\dfrac{\lambda-a \varepsilon}{2 \lambda \varepsilon^2}+O(1/z), \quad
\lambda_2(z)=\dfrac{\lambda}{\varepsilon}z-\dfrac{\lambda+a \varepsilon}{2 \lambda \varepsilon^2}+O(1/z),
\end{equation}
while the projectors are
\begin{equation}
\label{projector_z_big_JinXin}
\mathcal{P}_j(z)=R_j R_j^T+O(1/z), \quad j=1,2.
\end{equation}

\begin{remark}
\label{remark_zeta_infinity}
Notice that the term $O(1/z)$ in (\ref{eig_zeta_infty_JinXin}) could be singular in $\varepsilon$. However, from the previous discussion, the eigenvalues of $E(z)$ have a strictly negative real part. This implies that the coefficients of the even powers of $z$ in (\ref{eig_zeta_infty_JinXin}) have a negative sign, while the others are imaginary terms. Thus, $e^{\lambda_{1,2}(z)}$ are bounded in $\varepsilon$.
\end{remark}

\begin{proposition}
\label{Proposition_z_infinity}
We have the following decomposition near $z = \infty$:
\begin{equation}
\label{Decomposition_z_large}
E(z)=\lambda_1(z) \mathcal{P}_1(z)+\lambda_2(z) \mathcal{P}_2(z),
\end{equation}
with $\lambda_1(z), \lambda_2(z)$ in (\ref{eig_zeta_infty_JinXin}), and $\mathcal{P}_1(z), \mathcal{P}_2(z)$ in (\ref{projector_z_big_JinXin}).
\end{proposition}

\section{Green function estimates} \label{Section3}
\paragraph{Green function estimates near $z=0$}
We associate to (\ref{F_approximation_JinXin}) the parabolic equation
$$\partial_t w+a \partial_x w=(\lambda^2-a^2 \varepsilon^2)\partial_{xx} w.$$
We can write the explicit solution
\begin{equation}
\label{z_piccolo_JinXin}
g(t, x)= \dfrac{1}{2 \sqrt{(\lambda^2 - a^2 \varepsilon^2)\pi t}} \exp \Bigg\{ -\dfrac{(x-at)^2}{4(\lambda^2 - a^2 \varepsilon^2)t} \Bigg\}.
\end{equation}
This means that, for some $c_1, c_2>0$,
\begin{equation}
\label{z_piccolo_JinXin_estimate}
|g(t, x)| \le \dfrac{c_1}{\sqrt{t}} e^{-(x-at)^2/ct}, \qquad (t, x) \in \mathbb{R}^{+}\times \mathbb{R}, \qquad \forall \varepsilon>0.
\end{equation}
Now, recalling Proposition \ref{Proposition_z_small} and considering the approximation $\tilde{P}(z)$ in (\ref{tilde_Projector}) of the total projector $P(z)$ in (\ref{total_projector_approximation}),
$$e^{E(z)t}=\hat{g}(z) \tilde{P}(z)+R_{-}(z)e^{F_{-}(z)t}L_{-}(z)+\hat{R}_1(t, z),$$
where $\hat{g}(z)=-az-(\lambda^2-\varepsilon^2a^2)z^2,$ and $R_1(t,x)$ is a remainder term, we take the inverse of the Fourier transform of 
\begin{equation}
\label{K_dritto_JinXin_Fourier}
\hat{K}(z)=\hat{g}(z) \tilde{P}(z),
\end{equation}
which yields the expression of the first part of the Green function near $z=0$, i.e.
\begin{equation}
\label{K_dritto_JinXin}
K(t, x)=\left(\begin{array}{cc}
g(t, x) \quad & \varepsilon \sqrt{\lambda^2 -a^2 \varepsilon^2} \Bigg(\dfrac{d g(t,x)}{dx}\Bigg) \\\\
\varepsilon \sqrt{\lambda^2 -a^2 \varepsilon^2} \Bigg(\dfrac{d g(t,x)}{dx}\Bigg) \quad & \varepsilon^2 (\lambda^2 - a^2 \varepsilon^2) \Bigg(\dfrac{d^2 g(t,x)}{d^2 x}\Bigg)
\end{array}\right).
\end{equation}
Here, $\hat{K}(t, \xi)$ is the approximation of $\hat{\Gamma}(t, \xi)$ in (\ref{Formal_solution_JinXin}) for $|\xi| \approx 0.$
Thus, for $\xi \in [-\delta, \delta] $ with $\delta>0$ sufficiently small, we consider the following remainder term
\begin{equation}
\label{Remainder_1_JinXin}
\begin{aligned}
R_1(t, x)&=\dfrac{1}{2\pi}\int_{-\delta}^\delta (e^{E(i \xi)t}-e^{\hat{K}(t, \xi)t}) e^{i \xi x} \, d \xi \\
& = \dfrac{1}{2\pi}\int_{-\delta}^\delta e^{i\xi(x-at)-\xi^2(\lambda^2-a^2 \varepsilon^2)t}  (e^{O(\xi^3 t)}P(i\xi)-\tilde{P}(i\xi))  \, d\xi \\
& + \dfrac{1}{2 \pi} \int_{-\delta}^\delta R_{-}(i \xi) e^{F_{-}(i \xi)t} L_{-}(i \xi) e^{i\xi x} \, d\xi.
\end{aligned}
\end{equation}
We need an estimate for the remainder above. First of all, from (\ref{F_meno_approximation_JinXin}) and (\ref{Compl_total_projector_JinXin}),
\begin{align*}
\Bigg| \dfrac{1}{2 \pi} \int_{-\delta}^\delta R_{-}(i \xi) e^{F_{-}(i \xi)t} L_{-}(i \xi) e^{i\xi x} \, d\xi \Bigg| & \le \Bigg| \dfrac{1}{2 \pi} \int_{-\delta}^\delta P_{-}(i \xi) e^{(-1/\varepsilon^2+a i \xi +O(\xi^2))t} e^{i\xi x} \, d\xi \Bigg| \\\\
& \le C e^{-t/\varepsilon^2}
\end{align*}
for some constant $C$. 
Following \cite{Bianchini},
$$|e^{O(\xi^3 t)} P(i \xi)-\tilde{P}(i \xi)| = |z^3|t e^{2 \mu |z|^2 t} \left(\begin{array}{cc}
O(1) \quad &O( \varepsilon ) |z| \\
O(\varepsilon) |z| \quad &  O(\varepsilon^2) |z|^2
\end{array}\right),$$
for a constant $\mu > 0$. 
This way, 
$$R_1(t, x) = e^{-(x-at)^2/(ct)} \left(\begin{array}{cc}
O(1)(1+t)^{-1} \quad &  O(\varepsilon)  (1+t)^{-3/2} \\
O(\varepsilon)  (1+t)^{-3/2}  \quad & O(\varepsilon^2) (1+t)^{-2}
\end{array}\right).$$

\paragraph{Green function estimates near $z=\infty$}
We associate to (\ref{eig_zeta_infty_JinXin}) the following equations:
\begin{align*}
\partial_t w + \dfrac{\lambda}{\varepsilon} \partial_x w = -\dfrac{\lambda-a \varepsilon}{2 \lambda \varepsilon^2} w, \quad
\partial_t w - \dfrac{\lambda}{\varepsilon} \partial_x w = -\dfrac{\lambda+a \varepsilon}{2 \lambda \varepsilon^2} w.
\end{align*}
We can write explicitly the solutions
\begin{align*}
g_1(t, x)=\delta (x-\lambda t/ \varepsilon ) e^{-(\lambda-a \varepsilon)t/(2 \lambda \varepsilon^2) }, \quad
 g_2(t, x)=\delta (x+\lambda t/ \varepsilon ) e^{-(\lambda+a \varepsilon)t/(2 \lambda \varepsilon^2) }.
\end{align*}
Thus, 
$$|g_{j}(t, x)| \le C \delta( x \pm \lambda t/\varepsilon) e^{-c t/\varepsilon^2 }, \quad j=1, 2.$$
We determine the Fourier transform of the Green function for $|z|$ going to infinity, 
\begin{equation}
\label{K_storto_JinXin}
\hat{\mathcal{K}}(t, \xi) =  \exp \Bigg\{- i \dfrac{\lambda t \xi}{\varepsilon}   - \dfrac{(\lambda - a \varepsilon)t}{2 \lambda \varepsilon^2} \Bigg\}\mathcal{P}_1(\infty) + \exp \Bigg\{ i \dfrac{\lambda t \xi}{\varepsilon}   - \dfrac{(\lambda + a \varepsilon)t}{2 \lambda \varepsilon^2}  \Bigg\}\mathcal{P}_2(\infty).
\end{equation}
This way, from Proposition \ref{Proposition_z_infinity}, the remainder term here is
\begin{equation}\label{Remainder2_Jin_Xin}
R_2(t, x)=\dfrac{1}{2 \pi} \int_{|\xi| \ge N} (e^{E(i \xi)t}-\hat{\mathcal{K}}(t, \xi)) e^{i\xi x} \, d\xi, \qquad \text{and}
\end{equation}
\begin{align*}
|R_2| & \le \dfrac{1}{2 \pi} \Bigg|\int_{|\xi| \ge N} e^{i\xi (x- \lambda t/\varepsilon)-(\lambda - a \varepsilon)t/(2 \lambda \varepsilon^2)} \cdot (e^{O(1)t/(i\xi) + O(1)t/\xi^2} \mathcal{P}_{1}(i\xi)-\mathcal{P}_{1}(\infty)) \, d\xi \Bigg|\\
& + \dfrac{1}{2 \pi} \Bigg|\int_{|\xi| \ge N} e^{i\xi (x+\lambda t/\varepsilon)-(\lambda+a \varepsilon)t/(2 \lambda \varepsilon^2)} \cdot (e^{O(1)t/(i\xi) + O(1)t/\xi^2} \mathcal{P}_{2}(i\xi)-\mathcal{P}_{2}(\infty)) \, d\xi \Bigg|.
\end{align*}
Following \cite{Bianchini} and thanks to Remark \ref{remark_zeta_infinity},
\begin{align*}
|R_2(t, x)| &\le Ce^{-ct/\varepsilon^2} \Bigg[ \Bigg| \int_{|\xi| \ge N} \dfrac{e^{i\xi(x\pm \lambda t/\varepsilon)}}{\xi} \, d\xi \Bigg| + \int_{|\xi|\ge N} \dfrac{1}{\xi^2} \, d\xi\Bigg] \le Ce^{-ct/\varepsilon^2}.
\end{align*}

\paragraph{Remainders in between}
\label{Remainders_between}
Until now, we studied the Green function of the linearized diffusive Jin-Xin system for $z \approx 0,$ which yields the parabolic kernel $\hat{K}$ in (\ref{K_dritto_JinXin_Fourier}), and for $z \approx \infty,$ so obtaining $\hat{\mathcal{K}}$ in (\ref{K_storto_JinXin}). In these two cases, we also provided estimates for the remainder terms:
\begin{itemize}
\item $R_1$ in (\ref{Remainder_1_JinXin}) for the parabolic kernel $K$ for $|\xi| \le \delta,$ with $\delta$ sufficiently small;

\item $R_2$ in (\ref{Remainder2_Jin_Xin}) for the transport kernel $\mathcal{K}$ for $|\xi| \ge N,$ with $N$ big enough.
\end{itemize}
It remains to estimate the last remainder terms, namely the parabolic kernel $K$ for $|\xi| \ge \delta, \; t \ge 1,$ the transport kernel $\mathcal{K}$ for $|\xi| \le N,$ and the kernel $E(z)$ for $\delta \le |\xi| \le N.$
\paragraph{Parabolic kernel $K(t, x)$ for $|\xi| \ge \delta,$ $\delta <<1$ }
Let us define
\begin{equation}
\label{Remainder_3_Jin_Xin}
R_3(t, x)=\dfrac{1}{2 \pi} \int_{|\xi| \ge \delta} \hat{K}(t, \xi) e^{i\xi x}\, d\xi.
\end{equation}
Thus, from (\ref{K_dritto_JinXin_Fourier}), for $t \ge 1,$
\begin{align*}
|R_3(t, x)|&\le C \Bigg| \int_{|\xi|\ge \delta} e^{i\xi(x-at)} e^{-(\lambda^2-a^2 \varepsilon^2)\xi^2 t} \tilde{P}(i \xi) \, d\xi \Bigg|\\
&\le \dfrac{C e^{-t/C}}{\sqrt{t}}\left(\begin{array}{cc}
O(1) \quad & O( \varepsilon)   \\
O(\varepsilon)  \quad & O( \varepsilon^2) 
\end{array}\right).
\end{align*}
\paragraph{Transport kernel for $|\xi| \le N$}
Set
\begin{equation}
\label{Remainder_4_Jin_Xin}
R_4(t, x)=\dfrac{1}{2\pi} \int_{|\xi| \le N} \hat{\mathcal{K}}(t, \xi) e^{i\xi x} \, d\xi,
\end{equation}

and, from (\ref{K_storto_JinXin}),
\begin{align*}
|R_4(t, x)| & \le C e^{-(\lambda+|a|\varepsilon)t/(2 \lambda \varepsilon^2)} \sum \Bigg| \int_{-N}^N e^{i\xi (x\pm \lambda t/\varepsilon)} d\xi \Bigg| \\ 
&\le  C e^{-ct/\varepsilon^2} \min \Bigg\{ N, \dfrac{1}{|x\pm \lambda t/\varepsilon|} \Bigg\}.
\end{align*}
\paragraph{Kernel $E(z)$ for $\delta \le |\xi| \le N$}
Finally, we set
$$R_5(t, x)=\dfrac{1}{2\pi}\int_{\delta \le |\xi| \le N} e^{E(i\xi)t} e^{i\xi t} d\xi.$$
Differently from \cite{Bianchini}, here we cannot simply apply the (SK) condition, as mentioned in the Introduction, and a further analysis is needed.
The eigenvalues of $E(i\xi)=-i\xi A-B$ are expressed here:
\begin{equation}
\label{Eig_SK}
\lambda_{1/2}=\dfrac{1}{2 \varepsilon^2}\Bigg(-1 \pm \sqrt{1-4\varepsilon^2(ia\xi +\lambda^2\xi^2)}\Bigg)=\dfrac{- 2(a i \xi + \lambda^2 \xi^2)}{1 \pm \sqrt{1-4\varepsilon^2(a i \xi + \lambda^2 \xi^2)}}.
\end{equation}
By using the Taylor expansion for $\varepsilon \approx 0,$
$$\lambda_1= -\dfrac{ai \xi + \lambda^2 \xi^2}{1-\varepsilon^2 (a i \xi + \lambda^2 \xi^2)}, \quad \lambda_{2}=-\dfrac{1}{\varepsilon^2}.$$

Explicitly, denoting by 
$$\bigtriangleup=\sqrt{1-4\varepsilon^2 \xi (\lambda^2 \xi+ia)}, \;\;\;\; \Box_1=-1+\bigtriangleup+2ia\xi \varepsilon^2, \;\;\;\; \Box_2=1+\bigtriangleup-2ia\xi \varepsilon^2,$$
one can find that $e^{E(i\xi)t}$
\begin{equation*}
\begin{aligned}
&=\left(\begin{array}{cc}
\dfrac{e^{\lambda_2 t}}{2 \bigtriangleup}\Box_1+\dfrac{e^{\lambda_1 t}}{2\bigtriangleup}\Box_2 & -\dfrac{i\xi \varepsilon \sqrt{\lambda^2-a^2\varepsilon^2}(e^{\lambda_1 t}-e^{\lambda_2 t})}{\bigtriangleup}\\\\
-\dfrac{i\xi \varepsilon \sqrt{\lambda^2-a^2\varepsilon^2}(e^{\lambda_1 t}-e^{\lambda_2 t})}{\bigtriangleup} & \dfrac{e^{\lambda_1 t}}{2 \bigtriangleup}\Box_1+\dfrac{e^{\lambda_2 t}}{2\bigtriangleup}\Box_2
\end{array}\right), 
\end{aligned}
\end{equation*}
$$\text{where} \quad \Box_1=-1+\bigtriangleup=-2\varepsilon^2 \xi (\lambda^2 \xi+ia)+O(\varepsilon^2)=O(\varepsilon^2), \quad \Box_2=1+\bigtriangleup=O(1),$$
and, in terms of the singular parameter $\varepsilon,$ this yields
\begin{equation*}
e^{E(i\xi)t}=\left(\begin{array}{cc}
O(1) (e^{\lambda_1 t}+e^{\lambda_2 t}) & O(\varepsilon) (e^{\lambda_1 t}-e^{\lambda_2 t})\\\\
O(\varepsilon) (e^{\lambda_1 t}-e^{\lambda_2 t}) & e^{\lambda_1t}O(\varepsilon^2)+O(1)e^{\lambda_2t}
\end{array}\right).
\end{equation*}
Putting the calculations above all together and integrating in space with respect to the Fourier variable for $\delta \le |\xi| \le N,$ we get
\begin{equation}
\label{Remainder_5_Jin_Xin}
|R_5(t, x)| \le C \left(\begin{array}{cc}
O(1)e^{-t/C} & O(\varepsilon) e^{-t/C}\\\\
O(\varepsilon) e^{-t/C} & O(\varepsilon^2)e^{-t/C}+O(1)e^{-t/\varepsilon^2}
\end{array}\right).
\end{equation}
From (\ref{Remainder_1_JinXin}), (\ref{Remainder2_Jin_Xin}), (\ref{Remainder_3_Jin_Xin}), (\ref{Remainder_4_Jin_Xin}), (\ref{Remainder_5_Jin_Xin}), we denote the remainder by
\begin{equation}
\label{Remainder_total}
R(t)=R_1(t)+R_2(t)+R_3(t)+R_4(t)+R_5(t).
\end{equation}
The estimates above provide the following lemma.
\begin{lemma}
\label{Decomposition_lemma}
Let $\Gamma(t, x)$ be the Green function of the linear system (\ref{CD_real_JinXin_no_tilde}). We have the following decomposition:
$$\Gamma(t, x)=K(t, x)+\mathcal{K}(t,x)+R(t, x),$$
with $K(t, x), \mathcal{K}(t, x), R(t, x)$ in (\ref{K_dritto_JinXin}), (\ref{K_storto_JinXin}) and (\ref{Remainder_total}) respectively. Moreover, for some constants $c, C,$
\begin{itemize}
\item $ |K(t, x)| \le e^{-(x-at)^2/(ct)}  \left(\begin{array}{cc}
O(1)(1+t)^{-1} \quad &  O(\varepsilon)  (1+t)^{-3/2} \\
O(\varepsilon)  (1+t)^{-3/2}  \quad & O(\varepsilon^2) (1+t)^{-2}
\end{array}\right);$
\item $ |\mathcal{K}(t, x)| \le C e^{-ct/\varepsilon^2};$
\item 
$ 
\begin{aligned}
|R(t)| \le e^{-(x-at)^2/(ct)} & \quad \left(\begin{array}{cc}
O(1)(1+t)^{-1} \quad &  O(\varepsilon)  (1+t)^{-3/2} \\
O(\varepsilon)  (1+t)^{-3/2}  \quad & O(\varepsilon^2) (1+t)^{-2}
\end{array}\right)\\
&+\left(\begin{array}{cc}
O(1) & O(\varepsilon) \\
O(\varepsilon) & O(\varepsilon^2)
\end{array}\right)e^{-ct}+ Id \; e^{-ct/\varepsilon^2}.
\end{aligned}.$
\end{itemize}
\end{lemma}

\paragraph{Decay estimates}
Let us consider the solution to the Cauchy problem associated with the linear system (\ref{CD_real_JinXin_no_tilde}) and initial data $\textbf{w}_0,$
$$\hat{\textbf{w}}(t, \xi)=\hat{\Gamma}(t, \xi) \hat{\textbf{w}}_0(\xi)=e^{E(i\xi)t} \hat{\textbf{w}}_0(\xi).$$ By using the decomposition provided by Lemma \ref{Decomposition_lemma}, we get the following theorem.
\begin{theorem}
\label{Decay_estimates_JinXin}
Consider the linear system in (\ref{CD_real_JinXin_no_tilde}), i.e.
$$\partial_t \textbf{w}+A\partial_x \textbf{w}=-B\textbf{w},$$
and let $Q_0=R_0L_0$ and $Q_{-}=R_{-}L_{-}$ as before, i.e. the eigenprojectors onto the null space and the negative definite part of $-B$ respectively. Then, for any function $\textbf{w}_0 \in L^1 \cap L^2 (\mathbb{R}, \mathbb{R}),$ the solution $\textbf{w}(t)=\Gamma (t) \textbf{w}_0$ to the related Cauchy problem can be decomposed as
$$\textbf{w}(t)=\Gamma(t) \textbf{w}_0=K(t)\textbf{w}_0+\mathcal{K}(t)\textbf{w}_0+R(t) \textbf{w}_0.$$
Moreover, for any index $\beta$, the following estimates hold:\begin{equation}
\label{K_dritto_beta_estimates_L0_JinXin}
\begin{aligned}
\|L_0 D^\beta K(t) \textbf{w}_0\|_0 &\le C \min\{1, t^{-1/4-|\beta|/2}\} \|L_0 \textbf{w}_0\|_{L^1}\\
&+C \varepsilon \min\{ 1, t^{-3/4-|\beta|/2}\} \| L_{-} \textbf{w}_0\|_{L^1},
\end{aligned}
\end{equation}
\begin{equation}
\label{K_dritto_beta_estimates_Lmeno_JinXin}
\begin{aligned}
\|L_{-} D^\beta K(t) \textbf{w}_0\|_0 &\le C \varepsilon \min\{1, t^{-3/4-|\beta|/2}\} \|L_{0} \textbf{w}_0\|_{L^1}\\
&+C \varepsilon^2 \min\{ 1, t^{-5/4-|\beta|/2}\} \|L_{-} \textbf{w}_0\|_{L^1},
\end{aligned}
\end{equation}
\begin{equation}
\label{K_storto_beta_estimates_JinXin}
\|D^\beta \mathcal{K}(t) \textbf{w}_0\|_0 \le C e^{-ct/\varepsilon^2}\|D^\beta \textbf{w}_0\|_0,
\end{equation}
\begin{equation}
\label{remainder_total_beta_estimate_L0}
\begin{aligned}
\|L_0 D^\beta R(t) \textbf{w}_0\|_0 & \le C \min\{1, t^{-1/4-|\beta|/2}\} \|L_0 \textbf{w}_0\|_{L^1}\\
&+C \varepsilon \min\{ 1, t^{-3/4-|\beta|/2}\} \| L_{-} \textbf{w}_0\|_{L^1}\\
&+C e^{-ct} \|L_0 \textbf{w}_0\|_{L^1}+C\varepsilon e^{-ct}\|L_{-}\textbf{w}_0\|_{L^1}+Ce^{-ct/\varepsilon^2} \|\textbf{w}_0\|_{L^1},
\end{aligned}
\end{equation}
\begin{equation}
\label{remainder_total_beta_estimate_Lmeno}
\begin{aligned}
\|L_{-} D^\beta R(t) \textbf{w}_0\|_0 &\le C \varepsilon \min\{1, t^{-3/4-|\beta|/2}\} \|L_{0} \textbf{w}_0\|_{L^1}\\
&+C \varepsilon^2 \min\{ 1, t^{-5/4-|\beta|/2}\} \|L_{-} \textbf{w}_0\|_{L^1}\\
&+ C \varepsilon e^{-ct} \|L_0 \textbf{w}_0\|_{L^1}+C\varepsilon^2 e^{-ct}\|L_{-}\textbf{w}_0\|_{L^1}+Ce^{-ct/\varepsilon^2} \|\textbf{w}_0\|_{L^1}.
\end{aligned}
\end{equation}
\end{theorem}
\begin{proof}
From Lemma \ref{Decomposition_lemma}, for some constants $c, C >0,$ and for an index $\beta,$ it holds
\begin{equation}
\label{parabolic_kernel_L2_estimate_JinXin}
\|D^\beta \mathcal{K}(t) \textbf{w}_0\|_0 \le C e^{-ct/\varepsilon^2} \|D^\beta \textbf{w}_0\|_0.
\end{equation}
On the other hand, the hyperbolic kernel (\ref{K_dritto_JinXin_Fourier}) can be estimated as
\begin{align*}
& |L_0 \widehat{K(t)\textbf{w}_0}|\le C e^{-c|\xi|^2 t} (|L_0 \hat{\textbf{w}}_0|+\varepsilon|\xi| |L_{-}\hat{\textbf{w}}_0|), \\
& |L_{-} \widehat{K(t)\textbf{w}_0}|\le C e^{-c|\xi|^2 t} (\varepsilon|\xi||L_0 \hat{\textbf{w}}_0|+\varepsilon^2|\xi|^2 |L_{-}\hat{\textbf{w}}_0|).
\end{align*}
This yields
\begin{align*}
\|L_0 K(t) \textbf{w}_0 \|_0^2 &\le C \int_0^\infty \int_{S^0} e^{-2c|\xi|^2t} (|L_0\hat{\textbf{w}}_0(\xi)|^2 + \varepsilon^2|\xi|^2 |L_{-}\hat{\textbf{w}}_0(\xi)|^2) \, d\zeta d\xi \\
& \le C \min\{1, t^{-1/2}\} \|L_0 \hat{\textbf{w}}_0\|_\infty^2 + C \varepsilon^2 \min \{1, t^{-3/2}\} \|L_{-}\hat{\textbf{w}}_0\|_\infty^2 \\
& \le C \min\{1, t^{-1/2}\} \|L_0 {\textbf{w}}_0\|_{L^1}^2 + C \varepsilon^2 \min \{1, t^{-3/2}\} \|L_{-}{\textbf{w}}_0\|_{L^1}^2,
\end{align*}
and 
\begin{align*}
\|L_{-}K(t)\textbf{w}_0\|_0^2 & \le C \int_0^\infty \int_{S^0} e^{-2c|\xi|t} (\varepsilon^2 |\xi|^2 |L_0 \hat{\textbf{w}}_0(\xi)|^2+\varepsilon^4 |\xi|^2 |L_{-}\hat{\textbf{w}}_0(\xi)|^2) \, d\zeta d \xi \\
& \le C \varepsilon^2 \min\{1, t^{-3/2}\} \|L_0 \textbf{w}_0\|_{L^1}^2 + C \varepsilon^4 \min\{1, t^{-5/2}\} \|L_{-}\textbf{w}_0\|_{L^1}^2.
\end{align*}
Besides, for every $\beta$ we multiply by $\xi^{2\beta}$ the integrand and we get
\begin{equation*}
\begin{aligned}
\|L_0 D^\beta K(t) \textbf{w}_0\|_0 &\le C \min\{1, t^{-1/4-|\beta|/2}\} \|L_0 \textbf{w}_0\|_{L^1}\\
&+C \varepsilon \min\{ 1, t^{-3/4-|\beta|/2}\} \| L_{-} \textbf{w}_0\|_{L^1},
\end{aligned}
\end{equation*}
\begin{equation*}
\begin{aligned}
\|L_{-} D^\beta K(t) \textbf{w}_0\|_0 & \le C \varepsilon \min\{1, t^{-3/4-|\beta|/2}\} \|L_{0} \textbf{w}_0\|_{L^1}\\
&+C \varepsilon^2 \min\{ 1, t^{-5/4-|\beta|/2}\} \|L_{-} \textbf{w}_0\|_{L^1}.
\end{aligned}
\end{equation*}
The estimates for $R(t)$ are obtained in a similar way.
\end{proof}
\section{Decay estimates and convergence} \label{Section4}
Consider the local solution $\textbf{w}$ to the Cauchy problem associated with (\ref{CD_real_JinXin}), where we drop the \emph{tilde}, and initial data $\textbf{w}_0$. The solution to the nonlinear problem (\ref{CD_real_JinXin}) can be expressed by using the Duhamel formula \begin{equation}
\label{Duhamel}
\begin{aligned}
\textbf{w}(t)&=\Gamma(t)\textbf{w}_0+\int_0^t \Gamma (t-s) (N(w_1(s))-DN(0)w_1(s)) \, ds \\
& = \Gamma(t)\textbf{w}_0 + \int_0^t \Gamma(t-s) \left(\begin{array}{c}
0 \\
\dfrac{h(w_1(s))}{\varepsilon \sqrt{\lambda^2 - a^2 \varepsilon^2}} 
\end{array}\right) \, ds \quad t \in [0, T^*].
\end{aligned}
\end{equation} From (\ref{Pn_JinXin}) and the formulas below, we recall that
$w_1=L_0\textbf{w}=(1-L_{-})\textbf{w}$ is the conservative variable, while $w_2=L_{-}\textbf{w}$ is the dissipative one. We remind the Green function decomposition given by Lemma \ref{Decomposition_lemma}. For the $\beta$-derivative, 
\begin{align*}
D^\beta \textbf{w}(t)&=D^\beta K(t) \textbf{w}(0)+\mathcal{K}(t)D^\beta \textbf{w}(0)+R(t) D^\beta \textbf{w}(0) \\
&+\int_0^{t/2} D^\beta K(t-s)  R_{-}L_{-} \left(\begin{array}{c}
0 \\
\dfrac{h(w_1(s))}{\varepsilon \sqrt{\lambda^2 - a^2 \varepsilon^2}} 
\end{array}\right) \, ds \\
&+\int_{t/2}^t K(t-s) R_{-} D^\beta L_{-} \left(\begin{array}{c}
0 \\
\dfrac{h(w_1(s))}{\varepsilon \sqrt{\lambda^2 - a^2 \varepsilon^2}} 
\end{array}\right) \, ds \\
&+\int_0^t \mathcal{K}(t-s)D^\beta \left(\begin{array}{c}
0 \\
\dfrac{h(w_1(s))}{\varepsilon \sqrt{\lambda^2 - a^2 \varepsilon^2}} 
\end{array}\right) \, ds \\
&+\int_0^{t/2} D^\beta R(t-s)R_{-}L_{-} \left(\begin{array}{c}
0 \\
\dfrac{h(w_1(s))}{\varepsilon \sqrt{\lambda^2 - a^2 \varepsilon^2}} 
\end{array}\right) \, ds \\
&+ \int_{t/2}^t R(t-s)R_{-} D^\beta L_{-} \left(\begin{array}{c}
0 \\
\dfrac{h(w_1(s))}{\varepsilon \sqrt{\lambda^2 - a^2 \varepsilon^2}} 
\end{array}\right) \, ds\\\\
&= D^\beta K(t) \textbf{w}(0)+\mathcal{K}(t)D^\beta \textbf{w}(0)+R(t)D^\beta \textbf{w}(0) \\
&+\int_0^{t/2} \left(\begin{array}{c}
D^\beta K_{12}(t-s) \\
D^\beta K_{22}(t-s)
\end{array}\right)\dfrac{h(w_1(s))}{\varepsilon \sqrt{\lambda^2 - a^2 \varepsilon^2}}  \, ds \\
&+\int_{t/2}^t   \left(\begin{array}{c}
K_{12}(t-s)\\
K_{22}(t-s)
\end{array}\right) D^\beta \dfrac{h(w_1(s))}{\varepsilon \sqrt{\lambda^2 - a^2 \varepsilon^2}}  \, ds \\
&+\int_0^t \mathcal{K}(t-s)D^\beta \left(\begin{array}{c}
0 \\
\dfrac{h(w_1(s))}{\varepsilon \sqrt{\lambda^2 - a^2 \varepsilon^2}} 
\end{array}\right) \, ds
\end{align*}
\begin{align*}
\quad \;\; &+\int_0^{t/2} \left(\begin{array}{c}
D^\beta R_{12}(t-s) \\
D^\beta R_{22}(t-s)
\end{array}\right)\dfrac{h(w_1(s))}{\varepsilon \sqrt{\lambda^2 - a^2 \varepsilon^2}}  \, ds \\\\
&+\int_{t/2}^t   \left(\begin{array}{c}
R_{12}(t-s)\\
R_{22}(t-s)
\end{array}\right) D^\beta \dfrac{h(w_1(s))}{\varepsilon \sqrt{\lambda^2 - a^2 \varepsilon^2}}  \, ds. 
\end{align*}
Notice that, from (\ref{K_dritto_JinXin}), $K_{12}, K_{22}$ are of order $\varepsilon$ and $\varepsilon^2$ respectively, and the same holds for
\begin{align*}
&R_{12}=O(\varepsilon)(1+t)^{-3/2} e^{-(x-at)^2/ct}+O(\varepsilon)e^{-ct}+O(1)e^{-ct/\varepsilon^2},\\
&R_{22}=O(\varepsilon^2)(1+t)^{-2} e^{-(x-at)^2/ct}+O(\varepsilon^2)e^{-ct}+O(1)e^{-ct/\varepsilon^2}.
\end{align*}
From the assumptions above, $f(u)=f(w_1)=aw_1+h(w_1),$  where $h(w_1)=w_1^2 \tilde{h}(w_1)$  for some function $\tilde{h}(w_1).$
Thus, by using the estimates of Theorem \ref{Decay_estimates_JinXin}, and recalling that $\|\cdot\|_{m}=\|\cdot\|_{H^{m}(\mathbb{R})},$ for $m=0,1, 2,$ ($H^0=L^2$), we have, for $j=1,2,$
\begin{align*}
\|\textbf{w}(t)\|_{m} & \le C \min \{ 1, t^{-1/4}\} \|\textbf{w}_0\|_{L^1} + Ce^{-ct/\varepsilon^2}\|\textbf{w}_0\|_{m} \\
& + C \int_0^t \min\{1, (t-s)^{-3/4}\} (\|w_1^2 \tilde{h}(w_1)\|_{L^1}+ \|w_1^2 \tilde{h}(w_1)\|_m) \, ds \\
& + C \int_0^t e^{-c(t-s)} \|w_1^2 \tilde{h}(w_1)\|_m \, ds\\
& + C \int_0^t \dfrac{1}{\varepsilon} e^{-c(t-s)/\varepsilon^2} \|w_1^2 \tilde{h}(w_1)\|_m \, ds.
\end{align*}
For $m$ big enough, 
\begin{align*}
\|\textbf{w}(t)\|_{m} & \le C \min \{ 1, t^{-1/4}\} \|\textbf{w}_0\|_{L^1} + Ce^{-ct/\varepsilon^2}\|\textbf{w}_0\|_{m} \\
& + \int_0^t \min\{1, (t-s)^{-3/4}\} C(|w_1|_\infty)  \|w_1\|_m^2 \, ds \\
& + \int_0^t e^{-c(t-s)} C(|w_1|_\infty) \|w_1\|_m^2 \, ds \\
& + \int_0^t \dfrac{1}{\varepsilon} e^{-c(t-s)/\varepsilon^2} C(|w_1|_\infty) \|w_1\|_j^2 \, ds.
\end{align*}
From (\ref{change_CD_JinXin}), we recall that
$\textbf{w}=\left(\begin{array}{c}
w_1\\\\
w_2
\end{array}\right)
=\left(\begin{array}{c}
u\\\\
\dfrac{\varepsilon(v- au)}{\sqrt{\lambda^2-a^2\varepsilon^2}}
\end{array}\right),$
and so, for $m=2,$
\begin{align*}
\|u(t)\|_2 + c \varepsilon \|v(t)-au(t)\|_2  & \le C \min \{ 1, t^{-1/4}\} (\|{u}_0\|_{L^1}+c\varepsilon \|v_0-au_0\|_{L^1}) \\
&+ e^{-ct/\varepsilon^2}(\|{u}_0\|_{2}+c\varepsilon \|v_0-au_0\|_{2})  \\
& +  \int_0^t \min\{1, (t-s)^{-3/4}\} C(|u|_\infty)\|u\|_2^2 \, ds \\
& + \int_0^t e^{-c(t-s)} C(|u|_\infty) \|u\|_2^2 \, ds \\
& + \int_0^t \dfrac{1}{\varepsilon}e^{-c(t-s)/\varepsilon^2} C(|u|_\infty) \|u\|_2^2 \, ds.
\end{align*}
Let us denote by
\begin{equation}
E_m=\max\{\|u_0\|_{L^1}+\varepsilon \|v_0-au_0\|_{L^1}, \|u_0\|_m+\varepsilon \|v_0-au_0\|_m\},  
\end{equation}
where, according to (\ref{approx_initial_data_JinXin}), 
$v_0=f(u_0)-\lambda^2 \partial_x u_0,$
and
\begin{equation}
\label{M0_JinXin}
M_0(t)=\sup_{0\le \tau \le t} \{ \max\{1, \tau^{1/4}\} (\|u(\tau)\|_2+\varepsilon \|v(\tau)-au(\tau)\|_2)\}.
\end{equation}
The first term of the right hand ride of the estimate above gives
\begin{align*}
C \min \{ 1, t^{-1/4}\} (\|{u}_0\|_{L^1}+c\varepsilon \|v_0-au_0\|_{L^1})&+ Ce^{-ct/\varepsilon^2}(\|{u}_0\|_{2}+c\varepsilon \|v_0-au_0\|_{2})\\
& \le C \min \{1, t^{-1/4}\} E_2.
\end{align*}
Besides,
$$C(|u|_\infty) \|u\|_2^2 \le C(|u|_\infty) \min\{1, s^{-1/2}\} (M_0(s))^2.$$

Thus, 
\begin{align*}
\|u(t)\|_2+\varepsilon \|v(t)-au(t)&\|_2  \le  C \min \{1, t^{-1/4}\} E_2 \\
&+ (M_0(t))^2 \int_0^t e^{-c(t-s)} c(|u|_\infty) \min \{1, s^{-1/2}\} \, ds \\
&+ (M_0(t))^2 \int_0^t \dfrac{1}{\varepsilon} e^{-c(t-s)/\varepsilon^2} c(|u|_\infty) \min \{1, s^{-1/2}\} \, ds \\
& + (M_0(t))^2 \int_0^t c(|u|_\infty) \min \{1, (t-s)^{-3/4}\} \min \{1, s^{-1/2}\} \, ds.
\end{align*}
From the Sobolev embedding theorem,
$$c(|u(s)|_\infty) \le c(\|u(s)\|_2) \le C \min \{1, s^{-1/4}\} M_0(s) \le C M_0(s).$$
This way,
\begin{align*}
\|u(t)\|_2+\varepsilon \|v(t)-au(t)\|_2 & \le  C \min \{1, t^{-1/4}\} E_2 \\
&+C(M_0(t))^3 \int_0^t e^{-c(t-s)} \min \{1, s^{-1/2}\} \, ds \\
&+C(M_0(t))^3 \int_0^t \dfrac{1}{\varepsilon} e^{-c(t-s)/\varepsilon^2} \min \{1, s^{-1/2}\} \, ds \\
& + C(M_0(t))^3 \int_0^t \min \{1, (t-s)^{-3/4}\} \min \{1, s^{-1/2}\} \, ds.
\end{align*}
Notice that
\begin{align*}
\dfrac{1}{\varepsilon} \int_0^t e^{-c(t-s)/\varepsilon^2} \min\{1, s^{-1/2}\}  \, ds &= \varepsilon  e^{-ct/\varepsilon^2}\int_0^{t/\varepsilon^2} e^{c \tau} \min \{1, \varepsilon\sqrt{\tau}\} \, d\tau\\ 
& \le \varepsilon  e^{-ct/\varepsilon^2}\int_0^{t/\varepsilon^2} e^{c \tau} \, d\tau\\ 
& = \dfrac{\varepsilon}{c} [1-e^{-ct/\varepsilon^2}] \\
& \le C\varepsilon.
\end{align*}

By using this inequality in the estimate above,
\begin{align*}
\|u(t)\|_2+\varepsilon \|v(t)-au(t)\|_2 & \le  C \min \{1, t^{-1/4}\} E_2 \\
&+ C(M_0(t))^3 \int_0^t e^{-c(t-s)} \min \{1, s^{-1/2}\} \, ds  \\
& + \varepsilon C(M_0(t))^3 \\
& + C(M_0(t))^3 \int_0^t \min \{1, (t-s)^{-3/4}\} \min \{1, s^{-1/2}\} \, ds.
\end{align*}

By applying usual lemmas on integration, as Lemma 5.2 in \cite{Bianchini},
we get the following inequality
\begin{align*}
M_0(t) \le  C(E_2 + (M_0(t))^3).
\end{align*}
\begin{remark}
Notice that this last estimate and the calculations above are different from \cite{Bianchini}, since here we are not assuming the global well-posedness of our model. Indeed we prove this in the following.
\end{remark}
Then, if $E_2$ is small enough,
$$M_0(t) \le C E_2, $$ i.e.
\begin{equation}
\label{s_estimate_JinXin}
\|u(t)\|_2+\varepsilon \|v(t)-au(t)\|_2 \le C \min\{1, t^{-1/4}\} E_2.
\end{equation}

By arguing as before and following \cite{Bianchini}, we have the theorem below.
\begin{proposition}
The following estimates hold, with $C$ a constant independent of $\varepsilon$,
\begin{equation}
\label{decay_estimate_beta_JinXin}
\|D^\beta \textbf{w}(t)\|_0 \le C \min \{1, t^{-1/4-|\beta|/2}\} E_{|\beta|+3/2},
\end{equation}
\begin{equation}
\label{decay_estimate_beta_dissipative_JinXin}
\|D^\beta w_2(t)\|_0 \le C \min \{1, t^{-3/4-|\beta|/2}\} E_{|\beta|+3/2},
\end{equation}
\begin{equation}
\label{decay_estimate_time_derivative_beta_JinXin}
\|D^\beta \partial_t \textbf{w} (t)\|_{0} \le C \min \{1, t^{-3/4-|\beta|/2}\} E_{|\beta|+5/2},
\end{equation}
\begin{equation}
\label{decay_estimate_time_derivative_dissipative_beta_JinXin}
\|D^\beta \partial_t w_2 (t)\|_{0} \le C \min \{1, t^{-5/4-|\beta|/2}\} E_{|\beta|+7/2}.
\end{equation}
\end{proposition}
This last result and estimate (\ref{s_estimate_JinXin}) prove Theorem \ref{uniform_global_existence}.
\begin{remark}
Notice that the estimates for the partial derivative in time of the solution (\ref{decay_estimate_time_derivative_beta_JinXin}), (\ref{decay_estimate_time_derivative_dissipative_beta_JinXin}) are uniform in $\varepsilon$ thanks to the particular form of the initial data (\ref{approx_initial_data_JinXin_f}). In fact, these estimates can be obtained by applying again the Duhamel formula as before and, similarly to (\ref{s_estimate_JinXin}), we get a bound for $\|D^\beta \partial_t \textbf{w} (t)\|_{0}$ which depends on $\|D^\beta \partial_t \textbf{w}|_{t=0}\|$. This norm is not singular in $\varepsilon$ thanks to the particular form of the initial data, as it is shown below.
The initial data satisfy (\ref{approx_initial_data_JinXin}), i.e.
$$v_0=f(u_0)-\lambda^2 \partial_x u_0=a u_0 + h(u_0) - \lambda^2\partial_x u_0.$$
In terms of the (C-D)-variable $\textbf{w},$ 
$$\begin{cases}
u=w_1, \\
v=aw_1+\dfrac{\sqrt{\lambda^2-a^2\varepsilon^2}}{\varepsilon}w_2,
\end{cases}$$
this gives the following relation:
\begin{equation}
\label{relation_initial_data_JinXin}
\dfrac{\sqrt{\lambda^2-a^2\varepsilon^2}}{\varepsilon}w_2^0=h(w_1^0)-\lambda^2 \partial_x w_1^0
\end{equation}
Using (\ref{relation_initial_data_JinXin}) in system (\ref{CD_real_JinXin}), this yields
$$\begin{cases}
\partial_t w_1|_{t=0}=-a\partial_x w_1^0 - \dfrac{\sqrt{\lambda^2-a^2\varepsilon^2}}{\varepsilon}\partial_x w_2^0, \\
\partial_t w_2|_{t=0}=- \dfrac{\sqrt{\lambda^2-a^2\varepsilon^2}}{\varepsilon}\partial_x w_1^0 + a\partial_x w_2^0+\dfrac{\lambda^2}{\varepsilon \sqrt{\lambda^2-a^2\varepsilon^2}}\partial_x w_1^0 \\
=\dfrac{a^2 \varepsilon}{\sqrt{\lambda^2-a^2\varepsilon^2}}\partial_x w_1^0+a\partial_x w_2^0.
\end{cases}$$
In terms of the original variable,
$$\begin{cases}
\partial_t w_1|_{t=0}=-\partial_x f(\bar{u}_0)+\lambda^2 \partial_{xx}\bar{u}_0, \\
\partial_t w_2|_{t=0}=\dfrac{a\varepsilon}{\sqrt{\lambda^2-a^2\varepsilon^2}}(\partial_x f (\bar{u}_0)-\lambda^2 \partial_{xx}\bar{u}_0).
\end{cases}$$
\end{remark}

\paragraph{Convergence in the diffusion limit and asymptotic behavior}
We perform the one dimensional Chapman-Enskog expansion. Recalling that
$$w_1=u, \qquad w_2=u_d,$$
where $u$ is the conservative variable and $u_d$ is the dissipative one, system (\ref{CD_real_JinXin}) is
$$\partial_t \left(\begin{array}{c}
u \\
u_d
\end{array}\right) + A \partial_x \left(\begin{array}{c}
u \\
u_d
\end{array}\right) = \left(\begin{array}{c}
0 \\
q(u)
\end{array}\right),$$
with $A$ in (\ref{matrices_CD_real_JinXin}) and $q(u)=-\dfrac{w_2}{\varepsilon^2}+\dfrac{h(w_1)}{\varepsilon \sqrt{\lambda^2-a^2\varepsilon^2}}.$
We consider the following nonlinear parabolic equation
\begin{equation*}
\partial_t u + a \partial_x u + \partial_x h(u) -(\lambda^2-a^2 \varepsilon^2) \partial_{xx} u=\partial_x S,
\end{equation*}
where 
\begin{equation}
\label{S_estimate}
S=\varepsilon \sqrt{\lambda^2-a^2\varepsilon^2}\{\partial_t u_d - a \partial_x u_d \}.
\end{equation}
The homogeneous equation is
\begin{equation}
\label{CE_JinXin}
\partial_t w_p + a \partial_x w_p + \partial_x h(w_p) -(\lambda^2-a^2 \varepsilon^2) \partial_{xx} w_p=0,
\end{equation}
and associated Green function is provided here
$$\Gamma_p(t)=K_{11}(t)+\tilde{\mathcal{K}}(t)+\tilde{R}(t),$$
with $K_{11}$ in (\ref{K_dritto_JinXin}). We take the difference between the conservative variable $u=w_1$ and $w_p$,
\begin{equation}
\label{difference}
\begin{aligned}
D^\beta (u(t)-w_p(t)) & = \int_0^{t/2}  D^\beta D (K_{11}(t-s)+\tilde{R}(t-s)) (h(w_p(s))-h(u(s))) \, ds \\
& + \int_0^{t/2} D^\beta D(K_{11}(t-s) + \tilde{R}(t-s)) S(s) \, ds \\
& + \int_{t/2}^t D(K_{11}(t-s)+\tilde{R}(t-s))D^\beta (h(w_p(s))-h(u(s))+S(s)) \, ds \\
& + \int_0^t \tilde{\mathcal{K}}(t-s) D^\beta D(h(w_p(s))-h(u(s))+S(s)) \, ds.
\end{aligned}
\end{equation}
By using (\ref{decay_estimate_beta_JinXin}), (\ref{decay_estimate_beta_dissipative_JinXin}), (\ref{decay_estimate_time_derivative_dissipative_beta_JinXin}), 
 we have
$$\|D^\beta S \|_0 \le C \varepsilon \min \{1, t^{-5/4-\beta/2}\} E_{\beta+1}.$$
Let us define, for $\mu \in [0, 1/2),$
\begin{equation}
\label{m_0_JinXin}
m_0(t)=\sup_{ \tau \in [0, t] } \{ \max \{1, \tau^{1/4+\mu} \} \|u(\tau)-w_p(\tau)\|_0 \}.
\end{equation}
For $\beta=0,$
\begin{align*}
\|u(t)-w_p(t)\|_0 & \le C E_1 m_0(t) \int_0^t \min \{1, (t-s)^{-3/4}\} \min \{1, s^{-1/2-\mu}\} \, ds \\
& + C \varepsilon E_3 \int_0^t \min \{1, (t-s)^{-3/4}\} \min \{1, s^{-1} \} \, ds \\
& +  {C(E_1 m_0(t) + {\varepsilon}E_4)}\int_0^t e^{-c(t-s)} \min \{1, s^{-5/4}\} \, ds \\
& \le C \min \{1, s^{-1/4-\mu} \} ( E_1 m_0(t) + \varepsilon E_1 + \varepsilon E_4+(1/2-\mu)^{-1} E_1 m_0(t)),
\end{align*}
i.e., if $E_1$ is small enough,
\begin{equation}
\label{m_0_estimate}
m_0(t) \le C \varepsilon E_4.
\end{equation}

Similarly, it can be proved by induction that if, for $\gamma < \beta,$ defining
\begin{equation}
\label{m_beta2_JinXin}
m_\beta (t)=\sup_{ \tau \in [0, t] } \{ \max \{1, \tau^{1/4+\mu+\beta/2} \} \|D^\beta (u(\tau)-w_p(\tau))\|_0 \},
\end{equation}
and assuming 
$m_\gamma (t) \le C(\mu) \varepsilon E_{\gamma+4},$
then
$$\|D^\beta (h(u(s))-h(w_p(s)))\|_{0} \le C \min \{1, t^{-1/2-\mu-\beta/2}\} (C(\mu)E_{\beta+1}E_{\beta+3}+E_1m_\beta(t)).$$
Using this inequality, (\ref{S_estimate}) and (\ref{m_0_estimate}) in (\ref{difference}), finally we get
$$m_\beta (t) \le C(\mu) \varepsilon E_{\beta+4},$$
which ends the proof of Theorem \ref{asymptotic_behavior_JinXin}.

\section*{Acknowledgments}
The author is grateful to Roberto Natalini for some helpful advices and encouragements, and to Enrique Zuazua for useful discussions.

\bibliographystyle{siamplain}
\bibliography{references}

\end{document}